\documentclass[a4paper,12pt,twoside,leqno]{article}
\usepackage{amssymb} 
\usepackage{amsfonts}
\usepackage{amsmath}
\usepackage[utf8x]{inputenc}
\usepackage{tikz,relsize} 
\usepackage[cmtip,arrow]{xy}
\usepackage{pb-diagram,pb-xy}

\usepackage{tikz-cd}

\usepackage{hyperref}
\usepackage[noblocks]{authblk}

\usepackage{amssymb, amscd}

\usepackage[T1]{fontenc}

\usepackage{marginnote}
\usepackage{color}

\usepackage{amsthm}
\usepackage{mathtools}
\usepackage{tikz}
\usepackage{enumerate}
\usepackage{multirow}
\usepackage{longtable}

\newcommand{\al}{\alpha}
\newcommand{\bt}{\beta}
\newcommand{\te}{\theta}

\def\ns#1{\mathbb{#1}}
\def\N{\ns{N}}
\def\Z{\ns{Z}}

\def\R{\ns{R}}
\def\C{\ns{C}}
\def\F{\ns{F}}
\usepackage[mathscr]{eucal}

\DeclareMathOperator{\GL}{GL}

\begin{document}
\numberwithin{equation}{section}
\newtheorem{thm}{Theorem}[section]
\newtheorem{lm}[thm]{Lemma}
\newtheorem{cor}[thm]{Corollary}
\newtheorem{prop}[thm]{Proposition}
\newtheorem{dfn}[thm]{Definition}

 \newtheorem*{thmm}{Theorem}
\newtheorem*{deff}{Definition}
 \newtheorem*{corr}{Corollary}

\newtheorem{df}{Definition}[section]
\newtheorem{ex}{Example}[section]

\newtheorem{rem}{Remark}[section]
\author[]{R. Lutowski, N. Petrosyan, J. Popko, \& A. Szczepa\'{n}ski}

\title{\bf Spin structures of  flat manifolds of diagonal type}

\date{}

\maketitle

\begin{abstract}\noindent For each integer $d$ at least two, we construct  non-spin closed oriented flat manifolds with holonomy group $\Z_2^d$  and with the property that all of their finite proper covers have a spin structure. Moreover, all such covers have trivial Stiefel-Whitney classes. 
\end{abstract}
\section{Introduction} 
In this paper, we shall give a characterisation of  spin structures on closed flat manifolds with a diagonal
holonomy representation. In general, it is a difficult problem to classifying spin
structures on oriented flat manifolds.
If one is successful in defining a spin structure, it naturally leads to the definition of spinor fields, a Dirac operator and $\eta$-invariants on the manifolds (see e.g.~\cite{CMR, MP}).

Until now, the main direction of  research has been  on the
relation between the existence of a spin structure and properties of the
holonomy group and its representation. For example,  an oriented flat manifold has a spin structure if and only if its cover corresponding to a 2-Sylow subgroup of the holonomy has a spin structure. Hence,  from this point of view, more 
interesting flat manifolds are the ones with 2-group holonomy. From this class of manifolds, the simplest to describe are the flat manifolds with
holonomy group isomorphic to an elementary abelian 2-groups with representation of  diagonal
type. In fact, one of the first example of an oriented flat manifold
without  a spin structure is of this type (see \cite{LS}). For more information on this, we refer the  reader to \cite[\S 6.3]{S}.

Let us recall that every closed flat Riemannian manifold $M$ can be realised as a quotient of a Euclidean space by a discrete subgroup of the group of isometries $\Gamma \subseteq \mathrm{Iso}(\mathbb R^n)$ called a  {\it Bieberbach group}. More explicitly, considering the isomorphism $\mbox{Iso}(\mathbb R^n)\cong \mathbb R^n
\rtimes \mathrm{O}(n)$, any element of $\Gamma$ acts on $\mathbb R^n$ by a rotation and by a translation  in a canonical way.


By the classical Bieberbach theorems (see \cite{bieb1, bieb2}), $\R^n\cap \Gamma$ is a lattice and the quotient $G=\R^n/(\R^n\cap \Gamma)$ is a finite group called the
{\it holonomy group} of $M$. This leads to an exact sequence:
$$0\rightarrow \mathbb Z^n\stackrel{\iota}{\longrightarrow} \Gamma \stackrel{\pi}{\longrightarrow}G\rightarrow 1$$
where $\pi$ is the quotient map.  $M$ is said to be of {\it diagonal (holonomy) type} if  the induced
representation $\rho :G\to \mbox{GL}(n,\mathbb Z)$
is diagonal.

It is worth noting that, given an orientable closed flat Riemannian manifold $M$ with the fundamental group $\Gamma =\pi_1(M, p)$, the quotient map
$\pi: \Gamma \to G$, can be obtained by constructing a real orthogonal representation $\Gamma$ defined by translating  a vector
of the tangent space $TM_p$ at a fixed point $p\in M$ by the parallel
vector field along a loop at $p\in M$.
Since parallel transports of a given vector along homotopic loops produce the same
resulting vector, this give us a well-defined representation of $\Gamma$ into $\mathrm{O}(n)$. Restricting to its image, we recover the quotient map $\pi: \Gamma \to G$.

It follows that the holonomy group of any finite cover of $M$  is a subgroup of $G$.

We denote by $\mathrm{Spin}(n)$ the universal covering group of $\mathrm{SO}(n)$. We also write $\lambda_n:\mathrm{Spin}(n)\to \mathrm{SO}(n)$ for the  covering homomorphism. A {\it spin structure} on a smooth orientable manifold $M$ is an equivariant lift of its orthonormal frame bundle via the covering $\lambda_n$.  Equivalently, $M$ has a spin structure if and only if  the second Stiefel-Whitney class $w_2(M)$ vanishes (see \cite[p.~33-34]{kirby}).

\begin{deff}\rm Given a closed oriented flat manifold $M=\R^n/\Gamma$. We say that $M$ is {\it minimal non-spin} if it is non-spin and every finite cover with the holonomy group that is a proper subgroup of $G$ has a spin structure. 
\end{deff}
Let us point out that every closed oriented flat manifold  with holonomy group $\Z_2$ has a spin structure  (see \cite[Theorem 3.1(3)]{HS}, {\cite[Proposition 4.2]{MP}}). For any $d\in \N$,  set {\small $$n(d) = {d+1\choose 2} + \left\{ \begin{array}{cl}
2&\mbox{$d = 0$ mod $2$}\\
1&\mbox{$d = 1$ mod $4$}\\
3&\mbox{$d = 3$ mod $4$}
\end{array}\right.$$}

\noindent Our main result is the following theorem.
\begin{thmm}\label{thmm}
For any integer $d \geq 2$,  there exists a closed oriented flat manifold $M_d$  of rank $n(d)$ with holonomy group $\Z_2^d$  with the second Stiefel-Whitney class $w_2(M_d)\ne 0$ such that every finite cover with the holonomy group that is a proper subgroup of $\Z_2^d$ has all vanishing Stiefel-Whitney classes.
\end{thmm}

\noindent The following corollary is immediate.

\begin{corr}\label{corr}
For any integer $d \geq 2$, there exists a closed oriented flat manifold of rank
$n(d)$ with holonomy group $\Z_2^d$ of diagonal type which is minimal  non-spin.
\end{corr}

\noindent This result is in stark contrast to the case of real Bott manifolds which in part motivated our discussion. Real Bott manifolds are  a special type of flat manifolds with diagonal holonomy. By a result of  A.~G\k{a}sior (see \cite[Theorem 1.2]{gacior}), it follows that a real Bott manifold with holonomy group of even $\Z_2$-rank has a spin structure if and only if all its finite covers with holonomy group $\Z_2^2$ have a spin structure.  Our examples {show} that the general case of diagonal flat manifolds is much more complicated.

\section{Characterising diagonal flat manifolds}

In this section we give a combinatorial description of diagonal flat manifolds. This language will be essential in our analysis of the Steifel-Whitney classes of such manifolds.

Suppose we have a short exact sequence of groups
\begin{equation}\label{hwes}
0\to\Z^n\stackrel{\iota}\longrightarrow\Gamma\stackrel{\pi}\longrightarrow G\to 1.
\end{equation}
\noindent


 \noindent We shall call $\Gamma$ {\it diagonal} or {\it diagonal type}  if the image of the induced representation $\rho:G\to \mathrm{GL}(n, \Z)$:
 $$\rho(g)(z)= \iota^{-1}(\gamma  \iota(z) \gamma^{-1}), \;\;\forall g \in
G, \;\pi(\gamma)=g, \gamma\in \Gamma, \forall  z\in \Z^n,$$ 
  is a subgroup of the group of diagonal matrices $D\cong\Z_2^n\subseteq \GL(n,\Z)$ where    
$$D = \{A=[a_{ij}]\in\GL(n, \Z)\hskip 2mm|\hskip 2mm a_{ij} = 0, i\neq j; a_{ii} = \pm 1, 1\leq i,j\leq n\}.$$ 
It follows that $G = \Z_2^k$ for some $1\leq k\leq n-1$.

Let $S^1$ be the unit circle in $\C$. As in  \cite{PS},
we consider the automorphisms $g_i:S^1\to S^1,$ given by
\begin{equation}\label{dictionary}
g_0(z) = z, \; g_1(z) = -z, \; g_2(z) = \bar{z}, \; g_3(z) = -\bar{z}, \; \; \forall z\in S^1.
\end{equation}
Equivalently, with the identification $S^1 = \R/\Z,$ for any $[t]\in \R/\Z$ we have:
\begin{equation}\label{dictionary1}
g_0([t]) = [t], \;\; g_1([t]) =
\bigg[t+\frac{1}{2}\bigg], \;\; g_2([t]) = [-t], \;\;  g_3([t]) =\bigg[-t+\frac{1}{2}\bigg].
\end{equation}
Let ${\cal D} = \langle g_i\mid i = 0,1,2,3\rangle.$ It is easy to see that ${\cal D} \cong \Z_2\times\Z_2$ and $g_3 = g_1 g_2.$ We define an action ${\cal D}^n$ on $T^n$ by
\begin{equation}\label{action}
(t_1,\dots,t_n)(z_1,\dots, z_n) = (t_1 z_1,\dots,t_n z_n),
\end{equation}
for $(t_1,\dots,t_n)\in {\cal D}^n$ and $(z_1,\dots,z_n)\in T^n = \underbrace{S^1\times\dots\times S^1}_n.$

\noindent Any subgroup $\Z_2^d\subseteq {\cal D}^n$ defines a $({d\times n} )$-matrix with entries in $ {\cal D}$ which in turn defines a matrix $A$ with entries in the set $S=\{0, 1, 2, 3\}$ under the identification $i \leftrightarrow g_i$,  $0\leq i \leq 3$. Note that we can add distinct rows of $A$ to obtain a row vector with entries in $S$. 
 
 We have the following characterisation of the action of $\Z_2^d$ on $T^n$ and the associated  orbit space
 $T^n/\Z_2^d$ via the matrix $A$.

\begin{lm}\label{char}  Let $\Z_2^d\subseteq {\cal D}^n$  and $A\in S^{d\times n}$. Then,
\begin{enumerate}[(i)]
\item the action of $\Z_2^d$ on $T^n$ is free if and only if there is  $1$ in the sum of any distinct collection of rows of $A$,
\item  $\Z_2^d$ is the holonomy group of $T^n/\Z_2^d$ if and only if there is either $2$ or $3$ in each row of $A$.
\end{enumerate}
\end{lm}
\noindent When the action of $\Z_2^d$ on $T^n$ defined by (\ref{action}) is free, we will say that the associated  matrix $A$ is {\it free} and we will call it the {\it defining matrix} of $T^n/\Z_2^d$. In addition, when $\Z_2^d$ is the holonomy group of $T^n/\Z_2^d$, we will say  that $A$ is {\it effective}.

The proof of the above lemma is routine and shall be omitted.

\section{Interpreting Stiefel-Whitney classes}
We use defining matrices of diagonal flat manifolds to express their characteristic algebras and Stiefel-Whitney classes using the language introduced in the previous section.

To simplify notation,  we identify $i \leftrightarrow g_i$ for $i = 0,1,2,3.$  Let us consider the epimorphisms:
\begin{equation}\label{pppp}
\al, \bt:{\cal D}\to\F_2 = \{0,1\},
\end{equation}
where the values of $\al$ and  $\bt$ on $\cal D$ are given in the following table:
\begin{center}
\begin{tabular}{|c|c|c|c|c|}
\hline
& $0$ & $1$ & $2$ & $3$\\
\hline\hline
$\al$ & $0$ & $1$ & $1$ & $0$\\
\hline
$\bt$ & $0$ & $1$ & $0$ & $1$\\
\hline
\end{tabular}
\end{center}
\vskip 1mm
\centerline{\footnotesize{Table 1: $\alpha$ and $\beta$ on $\mathcal D$}}
\medskip
For $j = 1,\dots ,n$, and $\Z_2^{d}\subseteq {\cal D}^n$ we define the epimorphisms:
\begin{equation}\label{proj}
\al_j:\Z_2^{d}\subseteq {\cal D}^n\stackrel{pr_j}\longrightarrow {\cal D}\stackrel{{\alpha}}\to\F_2, \;\; \bt_j:\Z_2^{d}\subseteq {\cal D}^n\stackrel{pr_j}\longrightarrow {\cal D}\stackrel{{\beta}}\to\F_2
\end{equation}
by:
$$\al_j(t_1,\dots,t_n) = \al(t_j), \;\; \bt_j(t_1,\dots ,t_n) = \bt(t_j).$$

\noindent Using definitions of $\alpha$, $\beta$ and the translations given in the equation \eqref{dictionary1}, we obtain the following lemma.
\begin{lm}\label{holon}Suppose a subgroup $\Z_2^{d}\subseteq {\cal D}^n$ acts freely on $T^n$.
Then a holonomy representation $\varphi \colon \Z_2^d \to \GL(n,\Z)$ of the flat manifold $T^n/\Z_2^{d}$ is given by
\[
\forall_{x \in \Z_2^d} \; \varphi(x) = \operatorname{diag}\big((-1)^{(\alpha_1+\beta_1)(x)}, \ldots, (-1)^{(\alpha_n+\beta_n)(x)}\big).
\]
\end{lm}

Since 
$ H^1(\Z_2^{d};\F_2)=Hom(\Z_2^{d},\Z_2) $ we can view $\al_j$ and $\bt_j$ as $1$-cocycles and define:
\begin{equation}\label{transgression}
\te_j = \al_j\cup\bt_j\in H^2(\Z_2^{d};\F_2),
\end{equation}
where $\cup$ denotes the cup product.
It is well-known that
$$H^{\ast}(\Z_2^{d};\F_2)\cong \F_{2}[x_1,\dots,x_{d}]$$ where $\{x_1,\dots, x_d\}$ is a basis of $H^{1}(\Z_2^{d},\F_2)$.
Hence, the elements
$\al_j$ and $\bt_j $ correspond to: 
\begin{equation}\label{bformula}
\al_j = \sum_{i=1}^{d} \al(pr_{j}(b_{i}))x_{i},\;\; \bt_j = \sum_{i=1}^{d} \bt(pr_{j}(b_{i}))x_{i}    \in \F_{2}[x_{1},\dots,x_{d}],
\end{equation}
where $\{b_1,\dots,b_{d}\}$ is the standard basis of $\Z_{2}^{d}$ and $j = 1,\dots,n$ (cf. \cite[Proposition 1.3]{CMR}).
\vskip 2mm
\noindent
Moreover, from definition of the matrix $A\in {S}^{d\times n}$
we can write (\ref{bformula}) and (\ref{transgression}) write as:
\begin{equation}\label{abs_theta}
\al_j = \sum_{i=1}^{d} \al(A_{i,j})x_{i},  \;\; \bt_j = \sum_{i=1}^{d} \bt(A_{i,j})x_{i}, \;\; \te^A_j {=\al_{j}\cup\bt_{j}}= \al_{j}\bt_{j} .  
\end{equation}

Next, we will make use of the Lyndon-Hochschild-Serre spectral sequence $\{E_{r}^{p,q}, d_r\}$  associated to the group extension of (\ref{hwes}).
Since  $\Gamma$ is of diagonal type, we have:
$$E_{2}^{p,q} \cong H^{p}(\Z_2^{d};\F_2)\otimes H^{q}(\Z^{n}; \F_2).$$
There is an
exact sequence: 
\begin{equation}\label{transg}
\begin{split}
0\to  H^1(\Z_2^{d};\F_2)&\stackrel{\pi^{\ast}}\longrightarrow  H^1(\Gamma;\F_2) 
\\
 & \stackrel{\iota^{\ast}}\longrightarrow H^1(\Z^n;\F_2)\stackrel{d_2}\longrightarrow H^2(\Z_2^{d};\F_2)\stackrel{\pi^{\ast}}\longrightarrow H^2(\Gamma;\F_2), 
\end{split}
\end{equation}
where $d_2$ is the transgression and $\pi^{\ast}$ is induced by the quotient map $\pi :\Gamma\to \Z_2^{d}$ (e.g.~\cite[Corollary 7.2.3]{evens}).
\begin{prop}\label{key_sw} Suppose $\Z_2^{d}$ acts freely and diagonally on $T^n$. Let $M=T^n/\Z_2^{d}$, $\Gamma=\pi_1(M)$ and consider the  associated to the group extension of (\ref{hwes}). Then
\begin{enumerate}[(i)]
\item $\te_l=d_2(\varepsilon_l)$, $\forall 1\leq l\leq n$ , where $\{\varepsilon_1,\ldots,\varepsilon_n\}$ is the basis of $H^1(\Z^n,\F_2)$ dual to the standard basis of $\Z^n\otimes \F_2$.
\item The total Stiefel-Whitney class of $M$ is $$w(M)=\pi^{\ast}\bigg(\prod_{j=1}^n(1+\al_j+\bt_j)\bigg)\in H^{\ast}(\Gamma;\F_2)=H^{\ast}(M;\F_2).$$
\end{enumerate}
\end{prop}
\begin{proof}
By Theorem 2.5(ii) and   Proposition 1.3 of \cite{CMR} and using \eqref{dictionary1}, it follows that
\[
d_2(\varepsilon_l) = \sum_{A_{il}=1} x_i^2 + \sum_{i\ne j} x_ix_j,
\]
where the second sum is taken for such $i,j$ that
\[
(A_{il},A_{jl}) \in \{ (1,2),(2,1),(1,3),(3,1),(3,2),(2,3) \}.
\]
On the other hand
\[
\te_l = \alpha_l \beta_l = \sum_{i=1}^d \alpha(A_{il})\beta(A_{il})x_i^2 + \sum_{1 \leq i < j \leq d} \big( \alpha(A_{il})\beta(A_{jl}) + \alpha(A_{jl})\beta(A_{il}) \big) x_ix_j.
\]
Comparing coefficients of the above two polynomials finishes the proof of (i).

For the second part of the proposition, note that the image of the holonomy representation $\varphi$ of $M$, defined in Lemma \ref{holon}, is a subgroup of the group $D$ of diagonal matrices of $\GL(n,\Z)$. Now, let $\{x_1',\ldots,x_n'\}$ be the standard basis of $H^1(D,\Z_2)$ (i.e.~$x_j'$ checks whether the $j$-th entry of the diagonal is $\pm 1$). Using Proposition 3.2 of \cite{CMR} ( see also (2.1) of \cite{LS}), we have:
\[
w(M) = \pi^*\bigg(\prod_{j=1}^n(1+\varphi^*(x_j'))\bigg).
\]
Furthermore, for every $1 \leq l \leq d$ and $1 \leq j \leq n$, we have
\[
\varphi^*(x_j')(b_l) = x_j'(\varphi(b_l)) = (\alpha_j+\beta_j)(b_l)
\]
and the result follows.
\end{proof}


We observe that by part (i) of Proposition \ref{key_sw} the image of the differential $d_2$ is the ideal   generated by $\te_j$-s:
$$\langle {Im } (d_2)\rangle=\langle \te_1, \dots, \te_n\rangle\subseteq \F_2[x_1,x_2,...,x_{d}].$$
Given $A\in {S}^{d\times n}$, using (\ref{abs_theta}), we will set $I_A=\langle \te^A_1, \dots, \te^A_n\rangle$ and call it the {\it characteristic ideal} of $A$.
The quotient  ${\cal C}_A = \F_2[x_1,...,x_d]/I_A$ will be the {\it characteristic algebra} of $A$. Whenever there is no confusion, we will suppress the subscripts.

\begin{cor}\label{char} Suppose $\Z_2^{d}$ acts freely and diagonally on $T^n$. There is a canonical homomorphism of graded algebras $\phi:{\cal C}\to H^{*}(T^n/\Z_{2}^{d};\F_2)$
such that $\phi([w])=w(T^n/\Z_{2}^{d})$ where $[w]$ is the class of 
\begin{equation}\label{abs_sw}
w=\prod_{j=1}^n(1+\al_j+\bt_j)\in \F_2[x_1,x_2,...,x_{d}].
\end{equation}
  Moreover, $\phi$ is a monomorphism in degree less that or equal to two.
\end{cor}
\begin{proof} This follows directly from the exact sequence (\ref{transg}), with $\phi$ induced by the algebra homomorphism $\pi^*:H^*(\Z_2^{d};\F_2)\rightarrow H^*(\Gamma;\F_2)$. \end{proof}


\begin{dfn}\label{sw(A)}\rm
Given a matrix $A\in {S}^{d\times n}$, we define the {\it Stiefel-Whitney class} of $A$, denoted $w(A)$, to be the class $[w] \in \mathcal C_A$ defined by  (\ref{abs_sw}). 
\end{dfn}

\begin{cor}\label{abs_sw_rel} Suppose $A\in {S}^{d\times n}$ is free and $T^n/\Z_{2}^{d}$ is the corresponding flat manifold. Then $\phi(w(A))=w(T^n/\Z_{2}^{d}).$
\end{cor}
Next, we derive several properties of the Stiefel-Whitney classes and of characteristic ideals which will be key to our discussion later on.

\begin{lm}\label{sw_pr}
Let $A\in {S}^{d\times m}$, $B\in {S}^{d\times n}$  and
$[A, B]\in {S}^{d\times (m + n)}$.  Then,
\begin{enumerate}[(i)]
\item $w([A,B]) = w(A) w(B);$
\item  $I_{[A,B]} = I_{A} + I_{B};$
\item if $j$-column of $A$ has only elements $\{0,2\}$ or $\{0,3\}$, then $\te_{j}^{A} = 0.$
\end{enumerate}
\end{lm}

\begin{proof} 
By Definition \ref{sw(A)}, we have:
$$w([A,B]) =\prod_{j=1}^{m}(1+\al_j+\bt_j)\prod_{j=m+1}^{m+n}(1+\al_j+\bt_j)\in \F_2[x_1,x_2,...,x_{d}]$$ where $$\al_j = \sum_{i=1}^{d} \al(A_{i,j})x_{i},  \;\; \bt_j = \sum_{i=1}^{d} \bt(A_{i,j})x_{i}, \;\;\; \forall \; 1\leq j\leq m$$
and
$$\al_j = \sum_{i=1}^{d} \al(B_{i,j})x_{i},  \;\; \bt_j = \sum_{i=1}^{d} \bt(B_{i,j})x_{i}, \;\;\; \forall \; m+1\leq j\leq m+n.$$
Therefore, $w([A,B]) = w(A) w(B).$

To prove (ii), recall that  $I_{[A,B]}=\langle \te^{[A,B]}_1, \dots, \te^{[A,B]}_n\rangle$ with $\te^{[A,B]}_j = \al_{j}\bt_{j}$.  Note that, $\te^{[A,B]}_j =\te^{A}_j$ when $1\leq j\leq m$ and $\te^{[A,B]}_j =\te^{B}_{j{-m}}$ when $m+1\leq j\leq m+n$. Hence. $I_{[A,B]} = I_{A} + I_{B}$. 

Part (iii) follows from that fact that $\te^{A}_j = \al_{j}\bt_{j}$ and that $\al_{j}=0$ on $\{0,2\}$ and $\bt_{j}=0$ on $\{0,3\}$.

\end{proof}

\section{Proof of main theorem}

To define minimal non-spin manifolds we will make use of the following matrices:
\begin{enumerate}[{1.}]
\item $A_{0} = \left[
\begin{array}{l}
I_{(d-1)}\\
r
\end{array}
\right ] \in {S}^{d\times (d-1)}  ,$
where $I_{(d-1)}$ is the identity matrix and $r=(1,\dots,1)$.
\item $A_1=\big[c_1, \dots, c_{d(d-1)/2}\big] \in {S}^{d\times d(d-1)/2}$  with columns $c_k=2e_i + 3e_j$ for all $i < j$ ordered in lexicographical order. Here, $e_i$ denotes the column vector with $1$ in the $i$-th coordinate and $0$ everywhere else.
\item Let $A = [A_0, A_1]$, $B = 2(e_1 + e_2 + ... + e_d) \in {S}^{d\times 1},$ and $C = 2e_1 \in {S}^{d\times 1}.$ 
\item Let $E$ be the free matrix
$$E = \left\{ \begin{array}{cc}
[A,B,C,C]& \mbox{$d = 0$ mod $2$}\\
\left[A,B,B\right]&\mbox{$d = 1$ mod $4$}\\
A& \mbox{$d = 3$ mod $4$}
\end{array}
\right.$$
\item Finally, let $F\in {S}^{d\times n(d)}$ be the free and effective matrix defined by 
$$ F = \left\{ \begin{array}{lc}
E & \mbox{$d\neq 3$ mod $4$}\\
\left[E,C,C,C,C\right] & \mbox{$d = 3$ mod $4$}
\end{array}\right.$$
\end{enumerate}

Note that $$n(d) = {d+1\choose 2} + \left\{ \begin{array}{cl}
2&\mbox{$d = 0$ mod $2$}\\
1&\mbox{$d = 1$ mod $4$}\\
3&\mbox{$d = 3$ mod $4$}
\end{array}\right.$$\\

Let $\sigma_i$ be the $i$-th {elementary} symmetric {polynomial} on variable $\{x_1, \dots ,x_d\}.$ Consider the ideal ${J}\subseteq \F_2[x_1, \dots ,x_d]$ defined by: 
$$\boxed{J=\{x_{i}^2 + x_{j}^2\mid i\neq j\}+ \{x_{i}x_{j}\mid x_i\neq x_j\}}$$
\begin{lm} The matrix $A$ is free, $I_A= {J}$ and $$w({A}) = \big[(1 + \sigma_1)^{d-1}\big]\in \mathcal C_A= \F_2[x_1, \dots ,x_d]/J.$$   
\end{lm}
\begin{proof} The matrix $A$ is clearly free by definition. To see that $J=I_A$,  note that, by Lemma \ref{sw_pr}, we have  $I_{A}=I_{A_0}+I_{A_1}$.  Recall that $I_{A_0}=\langle \te^{A_0}_1, \dots, \te^{A_0}_{d-1}\rangle$   with
$\te^{A_0}_l = \al_{l}\bt_{l}$ for $1\leq l\leq d-1$. Now, we have:
\begin{align*}
\te^{A_0}_l &= \alpha_l \beta_l \\
&= \sum_{i=1}^d \alpha({A_0}_{il})\beta({A_0}_{il})x_i^2 + \sum_{1 \leq i < j \leq d} \big( \alpha({A_0}_{il})\beta({A_0}_{jl}) + \alpha({A_0}_{jl})\beta({A_0}_{il}) \big) x_ix_j\\
&=x_l^2+x_d^2.
\end{align*}

\noindent Similarly, $I_{A_1}=\langle \te^{A_1}_1, \dots, \te^{A_1}_{ d(d-1)/2}\rangle$  and $\te^{A_1}_l= x_ix_j$ for all $1 \leq i < j \leq d$. It is easy to see now that $J=I_{A_0}+I_{A_1}$.

To prove the last claim, we write:
\begin{align*}
w(A)= & w({A_0})w({A_1})  \\
= & w({A_1})\\
=&\big[\prod_{i < j}(1 + x_i + x_j)\big] \\
=&\big[1 + (d-1)\sigma_1 + {d(d-2)}\sigma_2 + {d-1\choose 2}\sigma_{1}^{2}\big].
\end{align*}
Since $\sigma_2 \in {J},$ it follows that 
\begin{align*}
w(A) =& \big[1 + (d-1)\sigma_1 +  {d-1\choose 2}\sigma_{1}^{2}\big]\\
=& \big[(1 + \sigma_1)^{d-1}\big]. 
\end{align*}

\end{proof}

Let $\varphi:P_2\to\F_2$ be the linear extension of the map given by
$\varphi(x_{i}^2) = 1$ and $\varphi(x_{i}x_{j}) = 0,$ for $i\neq j$, where $P_2$ denotes the space of homogenous polynomials of degree two. We make the following observations.
\begin{lm} Let $J_2 = \{x\in {J}\mid x \mbox{ is an element of degree } 2\}$. Then $J_2= Ker (\varphi).$ 
\end{lm}
\begin{proof} The spaces ${J}_2$ and  $Ker(\varphi)$ have the same basis.
\end{proof}

\begin{lm}\label{J}
We have 
\begin{enumerate}[(i)]
\item $I_B = 0$ and $w({B}) = [1 + \sigma_1].$
\item $I_C = 0$ and $w(C) = [1 + x_1].$
\item The matrix $E$ is free and and $I_E =J.$
\end{enumerate}
\end{lm}
\begin{proof}The first two claims follow directly from definitions. For the  proof of the last claim,  note that, by parts (i) and (ii), we have
that $I_E = I_A$. This finishes the claim, since $I_A=J$.
\end{proof}

\begin{prop}\label{total}
The flat manifold $M$ defined by the matrix $E$ has $w(M)=[1+x_1^2]\in \mathcal C$. In particular, $M$ is oriented, it does not have a spin structure and $w_i(M)=0$ for all $i>2$.
\end{prop}
\begin{proof} First, let us observe that $\dim {\cal C}_{2}^{E} = 1$ and ${\cal C}_{i}^{E} = 0$
for $i > 2.$  In fact, the first formula can be seen from the definition ${\cal C}_2= P_2/J_2$. The second formula follows from noting that any homogenous  polynomial in  $\F_2[x_1,\dots , x_d]$  of degree greater than two is in the ideal $J$.

Let us now calculate  the Stiefel-Whitney class $w(M)=w(E)$ of $M$.
We shall consider the following cases.

\smallskip

\noindent {\bf  Case 1} ($d$ is even). We have:
\begin{align*}
w(E)=&w(A)w(B)w(C)\\
=&[(1 + \sigma_1)^{d-1}(1 + \sigma_1)(1 + x_1)^2]\\
=&[(1 + \sigma_1)^d(1 + x_{1}^{2})]\\
=& [(1 + \sigma_{1}^{2})^{d/2}(1 + x_{1}^{2})].
\end{align*}
Since $d$ is even,  $\sigma_{1}^{2}$ is a sum of even number of squares. Hence,  $\sigma_{1}^{2}\in {J}$
and $w(E) = [1 + x_{1}^{2}].$
Therefore, $w_i(M) = 0$ for $i\ne 2$ and $w_2(M) =[x_{1}^{2}].$ But $x_{1}^{2}\notin J$
because $\varphi(x_{1}^{2})\neq 0.$

\smallskip

\noindent {\bf Case 2} ($d = 1$ mod $4$). We have:
\begin{align*}
w(E)=&w(A)w(B)^2\\
=&[(1 + \sigma_1)^{d + 1}]\\
=&[(1 + \sigma_1)^2]\\
=& [1 + \sigma_{1}^{2}].
\end{align*}
As above, $M$ is orientable and has no spin structure since $\varphi( \sigma_{1}^{2}) = d={1}$.

\smallskip

\noindent {\bf Case 3} ($d = 3$ mod $4$). We have:
\begin{align*}
w(E)=&[(1 + \sigma_1)^{d - 1}]\\
=&[(1 + \sigma_1)^2]\\
=& [1 + \sigma_{1}^{2}].
\end{align*}
Hence, as above, $M$ is orientable and has no spin structure.
\end{proof}

\begin{prop} \label{second_prop} Let $M=T^n/{\Z_2^d}$ be the flat manifold defined by the matrix $E$. Let $M'$ be a finite cover of $M$, $\Gamma=\pi_{1}(M)$, $\Gamma'=\pi_{1}(M')$ and $i:\Gamma'\to \Gamma$ be the inclusion corresponding to the covering. Suppose $\Gamma'/(\pi_1(T^n)\cap \Gamma')\cong \Z_2^k$ with $k<d$. Then $M'$ has trivial Stiefel-Whitney classes. 

\end{prop}
\begin{proof} Let ${\cal C} = \F_2[x_1,\dots ,x_d]/{J}$ be the characteristic algebra of $M$ (equivalently, of $E$)
and ${\cal C}'$ be the characteristic algebra of $M'$ {with characteristic ideal $I_{M'}$.} We claim that ${\cal C}_{l}' = 0$
for $l \geq 2.$ 

  To see this, we note that there is a commutative diagram with exact rows:

\[
\begin{tikzcd}
\pi_1(T^n)\cap \Gamma' \arrow{r}{\iota'} \arrow[hookrightarrow]{d}{} & \Gamma' \arrow{d}{i} \arrow{r}{\pi'} & \Z_2^k \arrow[hookrightarrow]{d}{j}\\
\pi_1(T^n) \arrow{r}{\iota} & \Gamma \arrow{r}{\pi}  & \Z_2^d
\end{tikzcd}
\]

\noindent Combining this with the equation (3.6), yields the commutative diagram:

\[
\begin{tikzcd}
H^1(\Gamma;\F_2) \arrow{d}{i^*} \arrow{r}{d_2} & H^2(\Z_2^{d};\F_2) \arrow[rightarrow]{d}{j^*}\\
H^1(\Gamma';\F_2) \arrow{r}{d'_2}  & H^2(\Z_2^{k};\F_2)
\end{tikzcd}
\]

\noindent This shows that $$j^*(J)=j^*(\langle Im(d_2)\rangle)\subseteq \langle Im(d'_2)\rangle =I_{M'}\subseteq H^*(\Z_2^{k};\F_2).$$
\noindent  Therefore, we get an induced epimorphism of algebras 
$j^{\ast}:{\cal C}\to {\cal C}'$. 

Recall that by Proposition \ref{total}, ${\cal C}_2=\{0, [x_1^2]\}$.  For any $y\in {\cal C}_1\smallsetminus\{0\}$ there
is $z\in {\cal C}_1$ such that $yz=[x_1^2]$.
Suppose otherwise and let $y = [a],$ $a\in \F_2[x_1,\dots,x_d]_1.$ If $y{\cal C}_1 = \{0\}$, then
for any $1\leq k\leq d,$  $ax_k\in {J}_2 =Ker(\varphi)$. This is impossible since $\varphi$ corresponds to
a non-degenerated symmetric two linear map.

Since $\text{dim}\hskip 1mm{\cal C}_1 > \text{dim}\hskip 1mm{\cal C}_{1}'$, there exists $y\in {\cal C}_1$
such that $j^{\ast}(y) = 0.$ We can find
$z\in {\cal C}_1$ so that $yz=[x_1^2]\in {\cal C}_2$.   Because $j^{\ast}$ is an epimorphism and ${\cal C}_2$ is one-dimensional,  $j^{\ast}(yz)$ 
generates ${\cal C}_{2}'.$ But  $j^{\ast}(yz) = j^{\ast}(y)j^{\ast}(z) = 0$ and therefore, ${\cal C}_{2}' = 0.$

Finally, since ${\cal C}_l = 0$ for $l > 2$ and $i^{\ast}$ is surjection, we obtain
the triviality of ${\cal C}_{l}'$ for $l > 2.$ This proves our claim and together with Proposition \ref{total} finishes the proof.
\end{proof}

We are now ready to prove our main result.

\begin{thm} Suppose $M$ is the flat manifold defined by the matrix $F$. Then, $M$ is orientable with holonomy group $\Z_2^d$,  $w_2(M)\ne 0$ and every finite cover with the holonomy group that is a proper subgroup of $\Z_2^d$ has all vanishing Stiefel-Whitney classes.
\end{thm}
\begin{proof} Since the matrix $F$ is effective, by Lemma \ref{char}, we know that the holonomy group is $\Z_2^d$. By Lemmas \ref{sw_pr} and \ref{J}, it follows that $I_F=J$, $\mathcal C_F= \mathcal C_E$, and $w(F)=w(E)=[1+x_1^2]$.  Hence, $M$ is orientable, but non-spin. The last claim follows from applying the proof of Proposition \ref{second_prop} to the manifold $M$ defined by the matrix $F$ in place of $E$.
\end{proof}

\bigskip

\bigskip

\noindent R.~Lutowski, J.~Popko, \& A.~Szczepa\'nski\\
{Institute of Mathematics, University of Gda\'{n}sk, Gda\'{n}sk, 80-952, Poland}\\
{\it E-mail-address}: rafal.lutowski@mat.ug.edu.pl, jpopko@mat.ug.edu.pl,\\
 aszczepa@mat.ug.edu.pl\\

\noindent{N.~Petrosyan}\\
{Mathematical Sciences, University of Southampton, SO17 1BJ, UK}\\
{\it E-mail-address}: n.petrosyan@soton.ac.uk

\end{document}